\documentclass[12pt]{article}

\usepackage{amsmath}
\usepackage{amsfonts}
\usepackage{latexsym}
\usepackage{amssymb}
\usepackage{curves}
\usepackage{graphicx}
\usepackage{amsthm,amscd, amssymb}

\newtheorem{thm}{Theorem}[section]
\newtheorem{pro}[thm]{Proposition}
\newtheorem{cor}[thm]{Corollary}
\newtheorem{lem}[thm]{Lemma}

\newcommand{\noin}{\noindent}
\newcommand{\comp}{\, {}_\circ \,}

\newcommand{\ppp}{{\vphantom{\gamma}}}

\newcommand{\SL}{\mbox{\rm{SL}}}

\newcommand{\Con}{\mbox{\rm{Con}}}
\newcommand{\Ind}{\mbox{\rm{Ind}}}
\newcommand{\Inf}{\mbox{\rm{Inf}}}
\newcommand{\Iso}{\mbox{\rm{Iso}}}
\newcommand{\Res}{\mbox{\rm{Res}}}

\newcommand{\can}{\mbox{\rm{can}}}
\newcommand{\con}{\mbox{\rm{con}}}
\newcommand{\id}{\mbox{\rm{id}}}
\newcommand{\ind}{\mbox{\rm{ind}}}
\renewcommand{\inf}{\mbox{\rm{inf}}}
\newcommand{\iso}{\mbox{\rm{iso}}}
\newcommand{\lin}{\mbox{\rm{lin}}}
\newcommand{\mob}{\mbox{\rm{m\"{o}b}}}
\newcommand{\res}{\mbox{\rm{res}}}

\newcommand{\oo}{\overline}

\newcommand{\frakK}{{\mathfrak K}}

\newcommand{\frakp}{{\mathfrak p}}

\newcommand{\FF}{{\mathbb{F}}}
\newcommand{\KK}{{\mathbb{K}}}
\newcommand{\ZZ}{{\mathbb{Z}}}

\newcommand{\cI}{{\mathcal I}}
\newcommand{\cJ}{{\mathcal J}}
\newcommand{\cK}{{\mathcal K}}
\newcommand{\cL}{{\mathcal L}}
\newcommand{\cP}{{\mathcal P}}
\newcommand{\cQ}{{\mathcal Q}}
\newcommand{\cR}{{\mathcal R}}
\newcommand{\cT}{{\mathcal T}}

\begin{document}

\title{A new canonical induction \\
formula for $p$-permutation modules}

\author{\large Laurence Barker \hspace{1in}
Hatice Mutlu
\\ \mbox{} \\
\normalsize Department of Mathematics \\
\normalsize Bilkent University \\
\normalsize 06800 Bilkent, Ankara \\
\normalsize Turkey}

\maketitle

\small

\begin{abstract}
\noin Applying Robert Boltje's theory of canonical induction,
we give a restriction-preserving formula expressing
any $p$-permutation module as a $\ZZ[1/p]$-linear
combination of modules induced and inflated from
projective modules associated with subquotient groups.
The underlying constructions include, for any given
finite group, a ring with a $\ZZ$-basis indexed by
conjugacy classes of triples $(U, K, E)$ where $U$
is a subgroup, $K$ is a $p'$-residue-free normal
subgroup of $U$ and $E$ is an indecomposable
projective module of the group algebra of $U/K$.

\smallskip
\noin 2010 {\it Mathematics Subject Classification:} 20C20.
\end{abstract}

\section{Introduction}
We shall be applying Boltje's theory of canonical
induction \cite{Bol98a} to the ring of $p$-permutation
modules. Of course, $p$ is a prime. We shall be
considering $p$-permutation modules for finite
groups over an algebraically closed field $\FF$ of
characteristic $p$. A review of the theory of
$p$-permutation modules can be found in
Bouc--Th\'{e}venaz \cite[Section 2]{BT10}.

A canonical induction formula for $p$-permutation
modules was given by Boltje \cite[Section 4]{Bol98b}
and shown to be $\ZZ$-integral. It expresses any
$p$-permutation module, up to isomorphism, as a
$\ZZ$-linear combination of modules induced from
a special kind of $p$-permutation module, namely,
the $1$-dimensional modules.

We shall be inducing from another special kind of
$p$-permutation module. Let $G$ be a finite
group. We understand all $\FF G$-modules to
be finite-dimensional. An indecomposable
$\FF G$-module $M$ is said to be {\bf exprojective}
provided the following equivalent conditions hold
up to isomorphism: there exists a normal subgroup
$K \unlhd G$ such that $M$ is inflated from a
projective $\FF G / K$-module; there exists
$K \unlhd G$ such that $M$ is a direct summand
of the permutation $\FF G$-module $\FF G / K$;
every vertex of $M$ acts trivially on $M$; some
vertex of $M$ acts trivially on $M$. Generally, an
$\FF G$-module $X$ is called {\bf exprojective}
provided every indecomposable direct summand
of $X$ is exprojective.

The exprojective modules do already play a
special role in the theory of $p$-permutation
modules. Indeed,  the parametrization of the
indecomposable $p$-permutation modules,
recalled in Section 2, characterizes any
indecomposable $p$-permutation module
as a particular direct summand of a module
induced from an exprojective module.

We shall give a $\ZZ[1/p]$-integral canonical
induction formula, expressing any
$p$-permutation $\FF G$-module, up to
isomorphism, as a $\ZZ[1/p]$-linear combination
of modules induced from exprojective modules.
More precisely, we shall be working with the
Grothendieck ring for $p$-permutation modules
$T(G)$ and we shall be introducing another
commutative ring $\cT(G)$ which, roughly
speaking, has a free $\ZZ$-basis consisting of
lifts of induced modules of indecomposable
exprojective modules. Letting $\KK$ be a field
of characteristic zero that is sufficiently large
for our purposes, we shall consider a ring
epimorphism $\lin_G : \cT(G) \rightarrow
T(G)$ and its $\KK$-linear extension $\lin_G :
\KK \cT(G) \rightarrow \KK T(G)$. The latter
is split by a $\KK$-linear map $\can_G :
\KK T(G) \rightarrow \KK \cT(G)$ which, as
we shall show, restricts to a $\ZZ[1/p]$-linear
map $\can_G : \ZZ[1/p] T(G) \rightarrow
\ZZ[1/p] \cT(G)$.

To motivate further study of the algebras
$\ZZ[1/p] \cT(G)$ and $\KK \cT(G)$, we
mention that, notwithstanding the formulas
for the primitive idempotents of $\KK T(G)$
in Boltje \cite[3.6]{Bol}, Bouc--Th\'{e}venaz
\cite[4.12]{BT10} and \cite{Bar}, the relationship
between those idempotents and the basis
$\{ [M_{P, E}^G] : (P, E) \in_G \cP(E) \}$ remains
mysterious. In Section 4, we shall prove that
$\KK \cT(G)$ is $\KK$-semisimple as well as
commutative, in other words, the primitive
idempotents of $\KK \cT(G)$ comprise a
basis for $\KK \cT(G)$. We shall also describe
how, via $\lin_G$, each primitive idempotent
of $\KK T(G)$ lifts to a primitive idempotent
of $\KK \cT(G)$.

\section{Exprojective modules}

We shall establish some general properties of
exprojective modules.

Given $H \leq G$, we write ${}_G \Ind {}_H$ and
${}_H \Res {}_G$ to denote the induction and
restriction functors between $\FF G$-modules
and $\FF H$-modules. When $H \unlhd G$, we write
${}_G \Inf {}_{G/H}$ to denote the inflation functor to
$\FF G$-modules from $\FF G/H$-modules. Given
a finite group $L$ and an understood isomorphism
$L \rightarrow G$, we write ${}_L \Iso {}_G$ to denote
the isogation functor to $\FF L$-modules from
$\FF G$-modules, we mean to say,
${}_L \Iso {}_G (X)$ is the $\FF L$-module obtained
from an $\FF G$-module $X$ by transport of
structure via the understood isomorphism.

Let us classify the exprojective $\FF G$-modules
up to isomorphism. We say that $G$ is
{\bf $p'$-residue-free} provided $G = O^{p'}(G)$,
equivalently, $G$ is generated by the Sylow
$p$-subgroups of $G$. Let $\cQ(G)$ denote the
set of pairs $(K, F)$, where $K$ is a
$p'$-residue-free normal subgroup of $G$ and
$F$ is an indecomposable projective
$\FF G / K$-module, two such pairs $(K, F)$
and $(K', F')$ being deemed the same provided
$K = K'$ and $F \cong F'$. We define an
indecomposable exprojective $\FF G$-module
$M_G^{K, F} = {}_G \Inf {}_{G/K} (F)$. By
considering vertices, we obtain the following
result.

\begin{pro} \label{2.1}
The condition $M \cong M_G^{K, F}$
characterizes a bijective correspondence
between:

\noin {\bf (a)} the isomorphism classes of
indecomposable exprojective
$\FF G$-modules $M$,

\noin {\bf (b)} the elements $(K, F)$ of $\cQ(G)$.
\end{pro}

In particular, for a $p$-subgroup $P$ of $G$, the
condition $E \cong {}_{N_G(P)} \Inf {}_{N_G(P)/P}(\oo{E})$
characterizes a bijective correspondence between,
up to isomorphism, the indecomposable exprojective
$\FF N_G(P)$-modules $E$ with vertex $P$ and the
indecomposable projective $\FF N_G(P)/P$-modules
$\oo{E}$. It follows that the
well-known classification of the isomorphism classes
of indecomposable $p$-permutation $\FF G$-modules,
as in Bouc--Th\'{e}venaz \cite[2.9]{BT10} for instance,
can be expressed as in the next result. Let $\cP(G)$
denote the set of pairs $(P, E)$ where $P$ is a
$p$-subgroup of $G$ and $E$ is an exprojective
$\FF N_G(P)$-module with vertex $P$, two such
pairs $(P, E)$ and $(P', E')$ being deemed the same
provided $P = P'$ and $E \cong E'$. We make
$\cP(G)$ become a $G$-set via the actions on the
coordinates. We define $M_{P, E}^G$ to be the
indecomposable $p$-permutation $\FF G$-module
with vertex $P$ in Green correspondence with $E$.

\begin{thm} \label{2.2}
The condition $M \cong M_{P, E}^G$
characterizes a bijective correspondence
between:

\noin {\bf (a)} the isomorphism classes of
indecomposable $p$-permutation
$\FF G$-modules $M$,

\noin {\bf (b)} the $G$-conjugacy classes of
elements $(P, E) \in \cP(G)$.
\end{thm}

We now give a necessary and sufficient condition
for $M_{P, E}^G$ to be exprojective.

\begin{pro} \label{2.3}
Let $(P, E) \in \cP(G)$. Let $K$ be the normal
closure of $P$ in $G$. Then $M_{P, E}^G$ is
exprojective if and only if $N_K(P)$ acts
trivially on $E$. In that case, $K$ is
$p'$-residue-free, $P$ is a Sylow $p$-subgroup
of $K$, we have $G = N_G(P)K$, the inclusion
$N_G(P) \hookrightarrow G$ induces an
isomorphism $N_G(P)/P \cong G/K$, and
$M_{P, E}^G \cong M_G^{K, F}$,
where $F$ is the indecomposable projective
$\FF G/K$-module determined, up to
isomorphism, by the condition $E \cong
{}_{N_G(P)} \Inf {}_{N_G(P)/P} \Iso {}_{G/K} (F)$.
\end{pro}

\begin{proof}
Write $M = M_{P, E}^G$. If $M$ is exprojective
then $K$ acts trivially on $M$ and, perforce,
$N_K(P)$ acts trivially on $E$.

Conversely,
suppose $N_K(P)$ acts trivially on $E$. Then
$P$, being a vertex of $E$, must be a Sylow
$p$-subgroup of $N_K(P)$. Hence, $P$ is a
Sylow $p$-subgroup of $K$. By a Frattini
argument, $G = N_G(P)K$ and we have an
isomorphism $N_G(P)/P \cong G/K$ as
specified. Let $X = {}_G \Ind {}_{N_G(P)} (E)$.
The assumption on $E$ implies that $X$ has
well-defined $\FF$-submodules
$$Y = \left\{ {\sum}_k k \otimes_{N_G(P)} x :
  x \in E \right\} \; , \;\;\;\; \;\;\;\; Y' = \left\{
  {\sum}_k k \otimes_{N_G(P)} x_k : x_k \in E,
  {\sum}_k x_k = 0 \right\}$$
summed over a left transversal $k N_K(P)
\subseteq K$. Making use of the
well-definedness, an easy manipulation
shows that the action of $N_G(P)$ on $X$
stabilizes $Y$ and $Y'$. Similarly, $K$
stabilizes $Y$ and $Y'$. So $Y$ and $Y'$
are $\FF G$-submodules of $X$. Since
$|K : N_K(P)|$ is coprime to $p$, we have
$Y \cap Y' = 0$. Since $|K : N_K(P)| =
|G : N_G(P)|$, a consideration of dimensions
yields $X = Y \oplus Y'$.

Fix a left transversal $\cL$ for $N_K(P)$ in $K$.
For $g \in N_G(P)$ and $\ell \in \cL$, we can
write ${}^g \ell = \ell_g h_g$ with $\ell_g \in \cL$
and $h_g \in N_K(P)$. By the assumption on $E$
again, $h_g x = x$ for all $x \in E$. So
$$g {\sum}_\ell \ell \otimes x =
  {\sum}_\ell {}^g \ell \otimes gx =
  {\sum}_\ell \ell_g \otimes gx
  = {\sum}_\ell \ell \otimes gx$$
summed over $\ell \in \cL$. We have shown that
${}_{N_G(P)} \Res {}_G (Y) \cong E$. A similar
argument involving a sum over $\cL$ shows that
$K$ acts trivially on $Y$. Therefore, $Y \cong
M_G^{K, F}$. On the other hand, $Y$ is
indecomposable with vertex $P$ and, by the
Green correspondence, $Y \cong M_{P, E}^G$.
\end{proof}

We shall be making use of the following
closure property.

\begin{pro} \label{2.4}
Given exprojective $\FF G$-modules
$X$ and $Y$, then the $\FF G$-module
$X \otimes_\FF Y$ is exprojective.
\end{pro}

\begin{proof}
We may assume that $X$ and $Y$ are
indecomposable. Then $X$ and $Y$ are,
respectively, direct summands of permutation
$\FF G$-modules having the form $\FF G / K$
and $\FF G / L$ where $K \unlhd G \unrhd L$.
By Mackey decomposition and the
Krull--Schmidt Theorem, every indecomposable
direct summand of $X \otimes Y$ is a direct
summand of $\FF G / (K \cap L)$.
\end{proof}

\section{A canonical induction formula}

Throughout, we let $\frakK$ be a class of finite
groups that is closed under taking subgroups. We
shall understand that $G \in \frakK$. In clarification
of a hypothesis imposed in Section 1, we define
$\KK$ to be a field of characteristic zero that splits
for all the groups in $\frakK$. We shall abuse
notation, neglecting to use distinct expressions to
distinguish between a linear map and its extension
to a larger coefficient ring.

Specializing some general theory in Boltje
\cite{Bol98a}, we shall introduce a commutative
ring $\cT(G)$ and a ring epimorphism $\lin_G :
\cT(G) \rightarrow T(G)$. We shall show that the
$\ZZ[1/p]$-linear extension $\lin_G : \ZZ[1/p]
\cT(G) \rightarrow \ZZ[1/p] T(G)$ has a splitting
$\can_G : \ZZ[1/p] T(G) \rightarrow \ZZ[1/p] \cT(G)$.
As we shall see, $\can_G$ is the unique splitting
that commutes with restriction and isogation.

To be clear about the definition of $T(G)$, the
Grothendieck ring of the category of
$p$-permutation $\FF G$-modules, we mention
that the split short exact sequences are the
distinguished sequences determining the
relations on $T(G)$. The multiplication on $T(G)$
is given by tensor product over $\FF$. Given a
$p$-permutation $\FF G$-module $X$, we write
$[X]$ to denote the isomorphism class of $X$. We
understand that $[X] \in T(G)$. By Theorem \ref{2.2},
$$T(G) = \bigoplus_{(P, E) \in_G \cP(G)}
  \ZZ [M_{P, E}^G]$$
as a direct sum of regular $\ZZ$-modules, the
notation indicating that the index runs over
representatives of $G$-orbits. Let $T^{\rm ex}(G)$
denote the $\ZZ$-submodule of $T(G)$ spanned
by the isomorphism classes of exprojective
$\FF G$-modules. By Proposition \ref{2.4},
$T^{\rm ex}(G)$ is a subring of $T(G)$. By
Proposition \ref{2.1}
$$T^{\rm ex}(G) = \bigoplus_{(K, F) \in_G
  \cQ(G)} \ZZ [M_G^{K, F}] \; .$$

For $H \leq G$, the induction and restriction
functors ${}_G \Ind {}_H$ and ${}_H \Res {}_G$
give rise to induction and restriction maps
${}_G \ind {}_H$ and ${}_H \res {}_G$ between
$T(H)$ and $T(G)$. Similarly, given $L \in \frakK$
and an isomorphism $\theta : L \rightarrow G$,
we have an evident isogation map ${}_L^\ppp
\iso {}_G^\theta : T(L) \leftarrow T(G)$. In particular,
given $g \in G$, we have an evident conjugation
map ${}_{{}^g H}^\ppp \con {}_H^g$. Boltje noted
that, when $\frakK$ is the set of subgroups of a
given fixed finite group, $T$ is a Green functor in
the sense of \cite[1.1c]{Bol98a}. For arbitrary
$\frakK$, a class of admitted isogations must
be understood, and the isogations and inclusions
between groups in $\frakK$ must satisfy the
axioms of a category. Granted that, then $T$ is
still a Green functor in an evident sense whereby
the conjugations replaced by isogations.

Following a construction in \cite[2.2]{Bol98a},
adaptation to the case of arbitrary $\frakK$
being straightforward, we form the $G$-cofixed
quotient $\ZZ$-module
$$\cT(G) = \big( \bigoplus_{U \leq G}
  T^{\rm ex}(U) \big)_G$$
where $G$ acts on the direct sum via the
conjugation maps ${}_{{}^g U}^\ppp \con {}_U^g$.
Harnessing the Green functor structure of $T$,
the restriction functor structure of $T^{\rm ex}$
and noting that $T^{\rm ex}(G)$ is a subring of
$T(G)$, we make $\cT$ become a Green functor
much as in \cite[2.2]{Bol98a}, with the evident
isogation maps. In particular, $\cT(G)$ becomes
a ring, commutative because $T(G)$ is
commutative. Given $x_U \in T^{\rm ex}(U)$,
we write $[U, x_U]_G$ to denote the image of
$x_U$ in $\cT(G)$. Any $x \in \cT(G)$ can be
expressed in the form
$$x = \sum_{U \leq_G G} [U, x_U]_G$$
where the notation indicates that the index runs
over representatives of the $G$-conjugacy classes
of subgroups of $G$. Note that $x$ determines
$[U, x_U]$ and $x_G$ but not, in general, $x_U$.
Let $\cR(G)$
be the $G$-set of pairs $(U, K, F)$ where
$U \leq G$ and $(K, F) \in \cQ(U)$. We have
$$\cT(G) = \bigoplus_{U \leq_G G, (K, F)
  \in_{N_G(U)} \cQ(U)} \ZZ [U, [M_U^{K, F}]]
  = \bigoplus_{(U, K, F) \in_G \cR(G)}
  \ZZ [U, [M_U^{K, F}]] \; .$$

We define a $\ZZ$-linear map
$\lin_G : \cT(G) \rightarrow T(G)$ such that
$\lin_G [U, x_U] = {}_G \ind {}_U (x_U)$.
As noted in \cite[3.1]{Bol98a}, the family $(\lin_G
: G \in \frakK)$ is a morphism of Green functors
$\lin : \cT \rightarrow T$. In particular, the map
$\lin_G : \cT(G) \rightarrow T(G)$ is a ring
homomorphism. Extending to coefficients in
$\KK$, we obtain an algebra map
$$\lin_G \: : \: \KK \cT(G) \rightarrow
  \KK T(G) \; .$$

Let $\pi_G : T(G) \rightarrow T^{\rm ex}(G)$
be the $\ZZ$-linear epimorphism such that
$\pi_G$ acts as the identity on $T^{\rm ex}(G)$
and $\pi_G$ annihilates the isomorphism class
of every indecomposable non-exprojective
$p$-permutation $\FF G$-module. By
$\KK$-linear extension again, we obtain a
$\KK$-linear epimorphism $\pi_G : \KK T(G)
\rightarrow \KK T^{\rm ex}(G)$. After
\cite[5.3a, 6.1a]{Bol98a}, we define a
$\KK$-linear map
$$\can_G \: : \: \KK T(G) \rightarrow \KK \cT(G)
  \: , \; \xi \mapsto \frac{1}{|G|} \sum_{U, V
  \leq G} |U| \, \mob(U, V) [U, {}_U \res {}_V
  (\pi_V( {}_V \res {}_G (\xi)))]_G$$
where $\mob()$ denotes the M\"{o}bius
function on the poset of subgroups of $G$.

\begin{thm} \label{3.1}
Consider the $\KK$-linear map $\can_G$.

\noin {\bf (1)} We have $\lin_G \comp \can_G
= \id_{\KK T(G)}$.

\noin {\bf (2)} For all $H \leq G$, we have
${}_H \res {}_G \comp \can_G =
\can_H \comp {}_H \res {}_G$.

\noin {\bf (3)} For all $L \in \frakK$ and
isomorphisms $\theta : L \leftarrow G$,
we have ${}_L^\ppp \iso {}_G^\theta \comp
\can_G^\ppp = \can_L^\ppp \comp
{}_L^\ppp \iso {}_G^\theta$.

\noin {\bf (4)} $\can_G[X] = [X]$ for all
exprojective $\FF G$-modules $X$.

Those four properties, taken together for all
$G \in \frakK$, determine the maps $\can_G$.
\end{thm}

\begin{proof}
In view of the discussion above, this follows from
the proof of \cite[5.3a]{Bol98a}.
\end{proof}

Parts (2) and (3) of the theorem can be interpreted
as saying that $\can_* : T \rightarrow \cT$ is a
morphism of restriction functors. It is not hard to
check that, when $\frakK$ is closed under the
taking of quotient groups, the functors $T$,
$T^{\rm ex}$, $\cT$ can be equipped with
inflation maps, and the morphisms $\lin_*$ and
$\can_*$ are compatible with inflation.

The latest theorem immediately yields the
following corollary.

\begin{cor} \label{3.2}
Given a $p$-permutation $\FF G$-module
$X$, then
$$[X] = \frac{1}{|G|} \sum_{U, V \leq G} |U| \,
  \mob(U, V) \, {}_G \ind {}_U \res {}_V
  (\pi_V({}_V \res {}_G[X])) \; .$$
\end{cor}

Given $p$-permutation $\FF G$-modules $M$
and $X$, with $M$ indecomposable, we write
$m_G(M, X)$ to denote the multiplicity of $M$ as
a direct summand of $X$. We write $\pi_G(X)$ to
denote the direct summand of $X$, well-defined up
to isomorphism, such that $[\pi_G(X)] = \pi_G[X]$.

\begin{lem} \label{3.3}
Let $\frakp$ be a set of primes. Suppose that,
for all $V \in \frakK$, all $p$-permutation
$\FF V$-modules $Y$, all $U \lhd V$ such that
$|V : U| \in \frakp$ and all $V$-fixed elements
$(K, F) \in \cQ(U)$, we have
$$m_U(M_U^{K, F}, \pi_U({}_U \Res {}_V (Y)))
  = \sum_{(J, E) \in \cQ(V)} m_V(M_U^{K, F},
  M_V^{J, E}) \, m_V(M_V^{J, E}, \pi_V(Y)) \; .$$
Then, for all $G \in \frakK$, we have $|G|_{\frakp'} \,
\can_G[Y] \in \cT(G)$, where $|G|_{\frakp'}$ denotes
the $\frakp'$-part of $|G|$.
\end{lem}

\begin{proof}
This is a special case of \cite[9.4]{Bol98a}. Indeed,
an easy inductive argument justifies our
imposition of the condition $|V : U| \in \frakp$
in place of the weaker condition that $V/U$
is a cyclic $\frakp$-group.
\end{proof}

We can now prove the $\ZZ[1/p]$-integrality
of $\can_G$.

\begin{thm} \label{3.4}
The $\KK$-linear map $\can_G$ restricts to
a $\ZZ[1/p]$-linear map $\ZZ[1/p] T(G)
\rightarrow$ \linebreak
$\ZZ[1/p] \cT(G)$.
\end{thm}

\begin{proof}
Let $\frakp$ be the set of primes distinct from
$p$. Let $V$, $Y$, $U$, $K$, $F$ be as in the
latest lemma. We must obtain the equality in the
lemma. We may assume that $Y$ is
indecomposable. If $Y$ is exprojective, then
$\pi_U({}_U \Res {}_V(Y)) \cong {}_U \Res {}_V
(Y)$ and $\pi_V(Y) \cong X$, whence the
required equality is clear. So we may assume
that $Y$ is non-exprojective. Then $\pi_V(Y)$
is the zero module. It suffices to show that
$M_U^{K, F}$ is not a direct summand of
${}_U \Res {}_V(Y)$. For a contradiction,
suppose otherwise. The hypothesis on $|V : U|$
implies that $U$ contains the vertices of $Y$.
So $Y \mid {}_V \Ind {}_U (X)$ for some
indecomposable $p$-permutation
$\FF U$-module $X$. Bearing in mind that
$(K, F)$ is $V$-stable, a Mackey decomposition
argument shows that $M_U^{K, F} \cong X$.
The $V$-stability of $(K, F)$ also implies that
$K \lhd V$. So
$$Y \mid {}_V \Ind {}_U \Inf {}_{U/K} (F) \cong
  {}_V \Inf {}_{V/K} \Ind {}_{U/K} (F) \; .$$
We deduce that $Y$ is exprojective. This is
a contradiction, as required.
\end{proof}

\begin{pro} \label{3.5}
The $\ZZ$-linear map $\lin_G : \cT(G)
\rightarrow T(G)$ is surjective. However, the
$\ZZ[1/p]$-linear map $\can_G : \ZZ[1/p] T(G)
\rightarrow \ZZ[1/p] \cT(G)$ need not restrict
to a $\ZZ$-linear map $T(G) \rightarrow \cT(G)$.
Indeed, putting $p = 3$ and $G = \SL_2(3)$,
then the isomorphically unique indecomposable
non-simple non-projective $p$-permutation
$\FF G$-module $Y$ satisfies
$3 [Q_8, (\can_G[Y])_{Q_8}] = 2[Q_8, X]$, where
$X$ is the isomorphically unique $2$-dimensional
simple $\FF Q_8$-module.
\end{pro}

\begin{proof}
Since every $1$-dimensional $\FF G$-module
is exprojective, the surjectivity of the
$\ZZ$-linear map $\lin_G$ follows from Boltje
\cite[4.7]{Bol98b}. Routine techniques confirm
the counter-example.
\end{proof}

\section{The $\KK$-semisimplicity of the
commutative algebra $\KK \cT(G)$}

Let $\cI(G)$ be the $G$-set of pairs $(P, s)$
where $P$ is a $p$-subgroup of $G$ and
$s$ is a $p'$-element of $N_G(P)/P$.
Choosing and fixing an arbitrary isomorphism
between a suitable torsion subgroup of
$\KK - \{ 0 \}$ and a suitable torsion subgroup
of $\FF - \{ 0 \}$, we can understand Brauer
characters of $\FF G$-modules to have values in
$\KK$. For a $p'$-element $s \in G$, we define
a species $\epsilon_{1, s}^G$ of $\KK T(G)$, we
mean, an algebra map $\KK T(G) \rightarrow \KK$,
such that $\epsilon_{1,s}^G[M]$ is the value, at
$s$, of the Brauer character of a $p$-permutation
$\FF G$-module $M$. Generally,
for $(P, s) \in \cI(G)$, we define a species
$\epsilon_{P, s}^G$ of $\KK T(G)$ such that
$\epsilon_{P, s}^G [M] =
\epsilon_{1,s}^{N_G(P)/P} [M(P)]$, where
$M(P)$ denotes the $P$-relative Brauer
quotient of $M^P$. The next result,
well-known, can be found in
Bouc--Th\'{e}venaz \cite[2.18, 2.19]{BT10}.

\begin{thm} \label{4.1}
Given $(P, s), (P', s') \in \cI(G)$, then
$\epsilon_{P, s}^G = \epsilon_{P', s'}^G$ if
and only if we have $G$-conjugacy $(P, s)
=_G (P', s')$. The set $\{ \epsilon_{P, s}^G :
(P, s) \in_G \cI(G) \}$ is the set of species of
$\KK T(G)$ and it is also a basis for the
dual space of $\KK T(G)$. The dual basis
$\{ e_{P, s}^G : (P, s) \in_G \cI(G) \}$ is the
set of primitive idempotents of $\KK T(G)$.
As a direct sum of trivial algebras over $\KK$,
we have
$$\KK T(G) = \bigoplus_{(P, s) \in_G \cI(G)}
  \KK e_{P, s}^G \; .$$
\end{thm}

Let $\cJ(G)$ be the $G$-set of pairs $(L, t)$ where
$L$ is a $p'$-residue-free normal subgroup of $G$
and $t$ is a $p'$-element of $G/L$. We define a
species $\epsilon_G^{L, t}$ of $\KK T^{\rm ex}(G)$
such that, given an indecomposable exprojective
$\FF G$-module $M$, then $\epsilon_G^{L, t}[M]
= 0$ unless $M$ is the inflation of an
$\FF G / L$-module $\oo{M}$, in which case,
$\epsilon_G^{L, t}$ is the value, at $t$, of the
Brauer character of $\oo{M}$. It is easy to show
that, given a $p$-subgroup $P \leq G$ and a
$p'$-element $s \in N_G(P)/P$, then
$\epsilon_{P, s}^G[M] = \epsilon_G^{L, t}[M]$ for
all exprojective $\FF G$-modules $M$ if and only
if $L$ is the normal closure of $P$ in $G$ and
$t$ is conjugate to the image of $s$ in $G/L$.
Hence, via the latest theorem, we obtain the
following lemma.

\begin{lem} \label{4.2}
Given $(L, t), (L', t') \in \cJ(G)$, then
$\epsilon_G^{L, t} = \epsilon_G^{L', t'}$ if and
only if $L = L'$ and $t =_{G/L} t'$, in other words,
$(L, t) =_G (L', t')$. The set $\{ \epsilon_G^{L, t}
: (L, t) \in_G \cJ(G) \}$ is the set of species of
$\KK T^{\rm ex}(G)$ and it is also a basis for
the dual space of $\KK T^{\rm ex}(G)$.
\end{lem}

Let $\cK(G)$ be the $G$-set of triples $(V, L, t)$
where $V \leq G$ and $(L, t) \in \cJ(V)$. Given
$(L, t) \in \cJ(G)$, we define a species
$\epsilon_{G, L, t}^G$ of $\KK \cT(G)$ such
that, for $x$ as in Section 3,
$$\epsilon_{G, L, t}^G(x) =
  \epsilon_G^{L, t}(x_G) \; .$$
Generally, for $(V, L, t) \in \cK(G)$, we define a
species $\epsilon_{V, L, t}^G$ of $\KK \cT(G)$
such that
$$\epsilon_{V, L, t}^G(x) =
  \epsilon_{V, L, t}^V({}_V \res {}_G (x)) \; .$$
Using Lemma \ref{4.2}, a straightforward
adaptation of the argument in \cite[2.18]{BT10}
gives the next result.

\begin{thm} \label{4.3}
Given $(V, L, t), (V', L', t') \in \cK(G)$, then
$\epsilon_{V, L, t}^G = \epsilon_{V', L', t'}^G$
if and only if $(V, L, t)$ \linebreak
$=_G (V', L', t')$. The set
$\{ \epsilon_{V, L, t}^G : (V, L, t) \in_G \cK(G) \}$
is the set of species of $\KK \cT(G)$ and it is
also a basis for the dual space of $\KK \cT(G)$.
The dual basis $\{ e_{V, L, t}^G : (V, L, t) \in_G
\cK(G) \}$ is the set of primitive idempotents of
$\KK \cT(G)$. As a direct sum of trivial algebras
over $\KK$, we have
$$\KK \cT(G) = \bigoplus_{(V, L, t) \in_G
  \cK(G)} \KK e_{V, L, t}^G \; .$$
\end{thm}

We have the following easy corollary on lifts of
the primitive idempotents $e_{P, s}^G$.

\begin{cor} \label{4.4}
Given $(P, s) \in \cI(G)$, then $e_{\langle P, s \rangle,
P, s}^G$ is the unique primitive idempotent $e$ of
$\KK \cT(G)$ such that $\lin_G(e) = e_{P, s}^G$.
\end{cor}


\begin{thebibliography}{00}

\bibitem[1]{Bar}
L.\ Barker, {\it An inversion formula for the primitive idempotents
of the trivial source algebra}, arXiv:1809.10984v1 [math.RT]
28 Sept 2018.
	
\bibitem[2]{Bol98a}
R.\ Boltje, {\it A general theory of canonical induction formulae}, 
J.\ Algebra, {\bf 206}, 293- 343 (1998).

\bibitem[3]{Bol98b}
R.\ Boltje, {\it Linear source modules and trivial source modules},
Proc.\ Sympos.\ Pure Math.\ {\bf 63}, 7-30 (1998).

\bibitem[4]{Bol}
R.\ Boltje, {\it Representation rings of finite groups, their species
and idempotent formulae}, J.\ Algebra (to appear).

\bibitem[5]{BT10}
S.\ Bouc, J.\ Th\'{e}venaz, {\it The primitive idempotents of the
$p$-permutation ring}, J.\ Algebra {\bf 323}, 2905-2915 (2010).

\end{thebibliography}
\end{document}